\documentclass[11pt]{article}
\usepackage[utf8]{inputenc}
\usepackage{amsthm}
\usepackage{amsfonts}
\usepackage{amsmath}
\usepackage{amssymb}
\usepackage{colonequals}
\usepackage{color}
\usepackage{epsfig}
\usepackage{url}
\usepackage{hyperref}
\newtheorem{theorem}{Theorem}
\newtheorem{problem}[theorem]{Problem}

\newtheorem{lemma}[theorem]{Lemma}

\newtheorem{observation}[theorem]{Observation}
\newcommand{\eps}{\varepsilon}
\newcommand{\ch}{\text{ch}}

\newcommand{\brm}[1]{\operatorname{#1}}

\usepackage{xcolor}
\usepackage[textsize=footnotesize]{todonotes}
\usepackage{xspace}							

\title{Flexibility of planar graphs of girth at least six\thanks{Work on this paper was supported by project 17-04611S (Ramsey-like aspects of graph coloring) of Czech Science Foundation.}}
\author{Zden\v{e}k Dvo\v{r}\'ak\thanks{Computer Science Institute (CSI), Charles University in Prague.
E-mail: \protect\href{mailto:rakdver@iuuk.mff.cuni.cz}{\protect\nolinkurl{rakdver@iuuk.mff.cuni.cz}}.}\and
Tom\'a\v{s} Masa\v{r}\'ik\thanks{Department of Applied Mathematics, Charles University in Prague.
E-mail: \protect\href{mailto:masarik@kam.mff.cuni.cz}{\protect\nolinkurl{masarik@kam.mff.cuni.cz}}. Partially supported by the project SVV–2017–260452.}\and
Jan Mus\'ilek\thanks{Department of Applied Mathematics, Charles University in Prague.
E-mail: \protect\href{mailto:stinovlas@kam.mff.cuni.cz}{\protect\nolinkurl{stinovlas@kam.mff.cuni.cz}}.}\and
Ond\v{r}ej Pangr\'ac\thanks{Computer Science Institute (CSI), Charles University in Prague.
E-mail: \protect\href{mailto:pangrac@iuuk.mff.cuni.cz}{\protect\nolinkurl{pangrac@iuuk.mff.cuni.cz}}.}}

\date{}

\begin{document}
\maketitle

\begin{abstract}
Let $G$ be a planar graph with a list assignment $L$.  Suppose a preferred color is given for some of the
vertices.  We prove that if $G$ has girth at least six and all lists have size at least three, then there exists
an $L$-coloring respecting at least a constant fraction of the preferences.
\end{abstract}

\section{Introduction}

In a proper graph coloring, we want to assign to each vertex of a graph one of a fixed number of colors
in such a way that adjacent vertices receive distinct colors.
Dvo\v{r}\'ak, Norin, and Postle~\cite{requests} (motivated by a similar notion considered by Dvo\v{r}\'ak and Sereni~\cite{dotri})
introduced the following graph coloring question:
If some vertices of the graph have a preferred color, is it possible to properly color the graph so that
at least a constant fraction of the preferences is satisfied?  As it turns out, this question is trivial in
the ordinary proper coloring setting with a bounded number of colors (the answer is always positive~\cite{requests}), but
in the list coloring setting this gives rise to a number of interesting problems.

A \emph{list assignment} $L$ for a graph $G$ is a function that to each vertex $v\in V(G)$ assigns
a set $L(v)$ of colors, and an \emph{$L$-coloring} is a proper coloring $\varphi$ such that $\varphi(v)\in L(v)$ for all $v\in V(G)$.
A graph $G$ is \emph{$k$-choosable} if $G$ is $L$-colorable from every assignment $L$ of lists of size at least $k$.
A \emph{weighted request} is a function $w$ that to each pair $(v,c)$ with $v\in V(G)$ and $c\in L(v)$
assigns a nonnegative real number.  Let $w(G,L)=\sum_{v\in V(G),c\in L(v)} w(v,c)$.
For $\eps>0$, we say that $w$ is \emph{$\eps$-satisfiable} if there exists an $L$-coloring $\varphi$ of $G$ such that
$$\sum_{v\in V(G)} w(v,\varphi(v))\ge\eps\cdot w(G,L).$$
An important special case is when at most one color can be requested at each vertex and all such colors have the
same weight (say $w(v,c)=1$ for at most one color $c\in L(v)$, and $w(v,c')=0$ for any other color $c'$):
A \emph{request} for a graph $G$ with a list assignment $L$ is a function $r$ with $\brm{dom}(r)\subseteq V(G)$
such that $r(v)\in L(v)$ for all $v\in\brm{dom}(r)$.  For $\eps>0$, a request $r$ is \emph{$\eps$-satisfiable}
if there exists an $L$-coloring $\varphi$ of $G$ such that $\varphi(v)=r(v)$ for at least $\eps|\brm{dom}(r)|$ vertices $v\in\brm{dom}(r)$.
Note that in particular, a request $r$ is $1$-satisfiable if and only if the precoloring given by $r$ extends to an $L$-coloring of $G$.
We say that a graph $G$ with the list assignment $L$ is \emph{$\eps$-flexible} if every request is $\eps$-satisfiable,
and it is \emph{weighted $\eps$-flexible} if every weighted request is $\eps$-satisfiable (of course, weighted $\eps$-flexibility implies $\eps$-flexibility).

Dvořák, Norin and Postle~\cite{requests} established basic properties of the concept and proved that several interesting graph classes are flexible:
\begin{itemize}
\item For every $d\ge 0$, there exists $\varepsilon>0$ such that $d$-degenerate graphs with assignments of lists of size $d+2$ are weighted $\eps$-flexible.
\item There exists $\varepsilon>0$ such that every planar graph with assignment of lists of size $6$ is $\eps$-flexible.
\item There exists $\varepsilon>0$ such that every planar graph of girth at least five with assignment of lists of size $4$ is $\eps$-flexible.
\end{itemize}
They also raised a number of interesting questions, including the following one.
\begin{problem}\label{prob-dnp}
Does there exists $\varepsilon>0$ such that every planar graph $G$ and assignment $L$ of lists of size
\begin{itemize}
\item[(a)] five in general,
\item[(b)] four if $G$ is triangle-free,
\item[(c)] three if $G$ has girth at least five
\end{itemize}
is (weighted) $\eps$-flexible?
\end{problem}
Let us remark that the proposed list sizes match the best possible bounds guaranteeing the existence of a coloring
from the lists~\cite{voigt1993,voigt1995}.  In~\cite{req-trfree}, we gave a positive answer to the part (b).

In this note, we consider coloring from lists of size three, cf. the part (c). We prove that girth at least six is sufficient to ensure
flexibility with lists of size three; this improves upon the result
of Dvořák, Norin and Postle~\cite{requests} who proved that girth at least $12$ is sufficient.
\begin{theorem}\label{thm:sixGirth}
There exists $\eps>0$ such that each planar graph of girth at least six with assignment of lists of size three is weighted $\eps$-flexible.
\end{theorem}

Note that while it is trivial to prove that planar graphs of girth at least six are 3-choosable, Theorem~\ref{thm:sixGirth} is
substantially more complicated to establish.  As the proof that planar graphs of girth five are 3-choosable~\cite{thomassen1995-34} is quite involved,
we suspect proving the part (c) of Problem~\ref{prob-dnp} (if true) would be hard.

\section{Preliminaries}

Let $H$ be a graph. For a positive integer $d$, a set $I\subseteq V(H)$ is \emph{$d$-independent} if the distance between any distinct vertices
of $I$ in $H$ is greater than $d$.  Let $1_I$ denote the characteristic function of $I$, i.e., $1_I(v)=1$ if $v\in I$ and $1_I(v)=0$ otherwise.
For functions that assign integers to vertices of $H$, we define addition and subtraction in the natural way,
adding/subtracting their values at each vertex independently.
For a function $f:V(H)\to\mathbb{Z}$ and a vertex $v\in V(H)$, let $f\downarrow	v$ denote the function such that $(f\downarrow v)(w)=f(w)$ for $w\neq v$
and $(f\downarrow v)(v)=1$.  A list assignment $L$ is an \emph{$f$-assignment} if $|L(v)|\ge f(v)$ for all $v\in V(H)$.

Suppose $H$ is an induced subgraph of another graph $G$.  For an integer $k\ge 3$, let $\delta_{G,k}:V(H)\to\mathbb{Z}$
be defined by $\delta_{G,k}(v)=k-\deg_G(v)$ for each $v\in V(H)$.  For another integer $d\ge 0$, we say that $H$ is a \emph{$(d,k)$-reducible}
induced subgraph of $G$ if
\begin{itemize}
\item[(FIX)] for every $v\in V(H)$, $H$ is $L$-colorable for every $((\deg_H+\delta_{G,k})\downarrow v)$-assignment $L$, and
\item[(FORB)] for every $d$-independent set $I$ in $H$ of size at most $k-2$, $H$ is $L$-colorable for every $(\deg_H+\delta_{G,k}-1_I)$-assignment $L$.
\end{itemize}
Note that (FORB) in particular implies that $\deg_H(v)+\delta_{G,k}(v)\ge 2$ for all $v\in V(H)$.
Before we proceed, let us give an intuition behind these definitions.  Consider any assignment $L_0$ of lists of size $k$
to vertices of $G$. The function $\delta_{G,k}$ describes how many more (or fewer) available colors each vertex has compared to its degree.
Suppose we $L_0$-color $G-V(H)$, and let $L'$ be the list assignment for $H$ obtained from $L_0$ by removing from the list of each vertex
the colors of its neighbors in $V(G)\setminus V(H)$.  In $L'$, each vertex $v\in V(H)$ has at least $\deg_H(v)+\delta_{G,k}(v)$
available colors, since each color in $L_0(v)\setminus L'(v)$ corresponds to a neighbor of $v$ in $V(G)\setminus V(H)$.
Hence, (FIX) requires that $H$ is $L'$-colorable even if we prescribe the color of any single vertex of $H$,
and (FORB) requires that $H$ is $L'$-colorable even if we forbid to use one of the colors on the $d$-independent set $I$.

The following lemma is implicit in Dvo\v{r}\'ak et al.~\cite{requests} and appears explicitly in~\cite{req-trfree}.
\begin{lemma}\label{lemma-redu}
For all integers $g,k\ge 3$ and $b\ge 1$, there exists $\eps>0$ as follows.  Let $G$ be a graph of girth at least $g$. 
If for every $Z\subseteq V(G)$,
the graph $G[Z]$ contains an induced $(g-3,k)$-reducible subgraph with at most $b$ vertices, then $G$ with any assignment of lists
of size $k$ is weighted $\eps$-flexible.
\end{lemma}

We also use the following well-known fact.
\begin{lemma}[Thomassen~\cite{Thomassen97}]\label{lemma-gallai}
Let $G$ be a connected graph and $L$ a list assignment such that $|L(u)|\ge \deg(u)$ for all
$u\in V(G)$.  If either there exists a vertex $u\in V(G)$ such that $|L(u)|>\deg(u)$, or
some $2$-connected component of $G$ is neither complete nor an odd cycle, then $G$ is $L$-colorable.
\end{lemma}

\section{Reducible configurations}

In view of Lemma~\ref{lemma-redu}, we aim to prove that every planar graph of girth at least $6$ contains
a $(3,3)$-reducible induced subgraph (with bounded number of vertices).
Note that in the case of $(d,3)$-reducibility, the $d$-independent set considered in (FORB) has size at most $1$.
Hence, (FORB) is implied by (FIX) and the additional assumption that $\deg_H(v)+\delta_{G,3}(v)-1\ge 1$,
i.e., that $\deg_H(v)\ge\deg_G(v)-1$, for every $v\in V(H)$.  In particular, the value of $d$ is not
important, and for brevity we will say \emph{$3$-reducible} instead of $(d,3)$-reducible.

\begin{observation}\label{obs-deg1}
In any graph $G$, a vertex of degree at most $1$ forms a $3$-reducible subgraph.
\end{observation}

Let $Z$ be a subgraph of a graph $G$. A \emph{$Z$-nice path} in $G$ is a path whose internal vertices have degree exactly three in $G$ and
do not belong to $Z$.
A vertex $u$ is \emph{$Z$-nicely reachable} from a vertex $v$ if there exists a $Z$-nice path from $v$ to $u$;
let $d_Z(u,v)$ denote the length of shortest such path.  If $u$ is not $Z$-nicely reachable from $v$, let us
define $d_Z(u,v)=\infty$.

\begin{lemma}\label{lemma-dist2}
Let $G$ be a graph, let $Z$ be a subgraph of $G$, and let $u,v\in V(G)\setminus V(Z)$ be distinct vertices of degree two.
If $u$ is $Z$-nicely reachable from $v$, then $G$ contains a $3$-reducible induced subgraph disjoint from $Z$ with at most $d_Z(u,v)+1$ vertices.
\end{lemma}
\begin{proof}
Let $H$ be a shortest $Z$-nice path from $u$ to $v$ in $G$.  Clearly, $|V(H)|=d_Z(u,v)+1$ and $H$ is an induced path in $G$ disjoint from $Z$.
We claim that $H$ is $3$-reducible.  Note that $\deg_H(x)=\deg_G(x)-1$ for every $x\in V(H)$, and thus it suffices to show
that $H$ satisfies (FIX).  For any vertex $x$ of $H$, consider a $(\deg_H+\delta_{G,3})\downarrow x$-assignment $L$.
Each vertex of $H$ other than $x$ has list of size at least two, and thus we can $L$-color the path $H$ greedily starting from $x$.
Hence, $H$ is $L$-colorable; and thus (FIX) holds.
\end{proof}

\begin{lemma}\label{lemma-distk}
Let $G$ be a graph, let $Z$ be a subgraph of $G$, let $v\in V(G)\setminus V(Z)$ be a vertex of degree $k\ge 4$,
and let $u_1, \ldots, u_{k-1}\in V(G)\setminus V(Z)$ be distinct vertices of degree two.
If all vertices $u_1$, \ldots, $u_{k-1}$ are $Z$-nicely reachable from $v$, then $G$ contains a $3$-reducible induced subgraph
disjoint from $Z$ with at most $1+\sum_{i=1}^{k-1}d_Z(u_i,v)$ vertices.
\end{lemma}
\begin{proof}
For $i=1,\ldots,k-1$, let $P_i$ be a shortest $Z$-nice path from $u_i$ to $v$ in $G$.
Let $H=P_1\cup\ldots\cup P_{k-1}$.  We can assume that every two of the paths intersect only in $v$
and that $H$ is an induced subgraph of $G$, as otherwise the claim follows
by Lemma~\ref{lemma-dist2}.  Note that $|V(H)|=1+\sum_{i=1}^{k-1}d_Z(u_i,v)$ and $V(H)\cap V(Z)=\emptyset$.

We claim that $H$ is $3$-reducible.  Note that $\deg_H(x)=\deg_G(x)-1$ for every $x\in V(H)$, and thus it suffices to show
that $H$ satisfies (FIX).  For any vertex $x$ of $H$, consider a $(\deg_H+\delta_{G,3})\downarrow x$-assignment $L$.
Each vertex of $H$ other than $x$ has list of size at least two, and thus we can $L$-color the tree $H$ greedily starting from $x$.
Hence, $H$ is $L$-colorable; and thus (FIX) holds.
\end{proof}

\begin{lemma}\label{lemma-2faces}
Let $G$ be a $2$-connected plane graph of girth at least six, and let $f_1$ and $f_2$ be distinct $6$-faces of $G$ sharing
at least one edge such that all vertices incident with $f_1$ or $f_2$ have degree at most three.
If $f_2$ is incident with a vertex of degree two, then $G$ contains a $3$-reducible induced subgraph
with at most $10$ vertices, all incident with $f_1$ or $f_2$.
\end{lemma}
\begin{proof}
Let $C_1$ and $C_2$ be the cycles bounding $f_1$ and $f_2$ in $G$, respectively.
Since $G$ has girth at least six, $C_1\cap C_2$ is a path.
Let $H$ be the subgraph of $G$ induced by vertices incident with $f_1$ or $f_2$; we have $|V(H)|\le 10$.

We claim that $H$ is $3$-reducible.  Note that $\deg_H(x)\ge\deg_G(x)-1$ for every $x\in V(H)$, and thus it suffices to show
that $H$ satisfies (FIX).  For any vertex $x$ of $H$, consider a $(\deg_H+\delta_{G,3})\downarrow x$-assignment $L$.
Let $c$ be a color in $L(x)$ and let $L'$ be the list assignment for $H-x$ obtained from $L$ by removing $c$ from the lists of neighbors of $x$.
Note $|L'(v)|\ge \deg_{H-x}(v)$ for all $v\in V(H-x)$, and $|L'(v)|>\deg_{H-x}(v)$ if $\deg_G(v)=2$.
Observe that $H-x$ is connected, and that if $\deg_G(x)=2$, then $H-x$ contains an even cycle.
Hence $H-x$ is $L'$-colorable by Lemma~\ref{lemma-gallai}, and thus $H$ is $L$-colorable.
We conclude that (FIX) holds.
\end{proof}

\begin{lemma}\label{lemma-3faces}
Let $G$ be a $2$-connected plane graph of girth at least six, and let $f_1$, $f_2$, and $f_3$ be distinct $6$-faces of $G$,
all incident with a common vertex.  If all vertices incident with these faces have degree at most three,
then $G$ contains a $3$-reducible induced subgraph with at most $13$ vertices, all incident with $f_1$, $f_2$, or $f_3$.
\end{lemma}
\begin{proof}
Let $H$ be the subgraph of $G$ induced by vertices incident with $f_1$, $f_2$, or $f_3$; we have $|V(H)|\le 13$.
Since $G$ has girth at least six, the boundaries of any two of the faces intersect in a path.
By Lemma~\ref{lemma-2faces}, we can assume that all vertices of $H$ have degree exactly three in $G$;
consequently, any two of the faces $f_1$, $f_2$, and $f_3$ share exactly one edge.

We claim that $H$ is $3$-reducible.  Note that $\deg_H(x)\ge\deg_G(x)-1$ for every $x\in V(H)$, and thus it suffices to show
that $H$ satisfies (FIX).  For any vertex $x$ of $H$, consider a $(\deg_H+\delta_{G,3})\downarrow x$-assignment $L$.
Let $c$ be a color in $L(x)$ and let $L'$ be the list assignment for $H-x$ obtained from $L$ by removing $c$ from the lists of neighbors of $x$.
Note $|L'(v)|\ge \deg_{H-x}(v)$ for all $v\in V(H-x)$.  Observe that $H-x$ is connected and contains an even cycle.
Hence $H-x$ is $L'$-colorable by Lemma~\ref{lemma-gallai}, and thus $H$ is $L$-colorable.
We conclude that (FIX) holds.
\end{proof}

\section{Discharging}

Let $G$ be a $2$-connected plane graph of girth at least six such that $|V(G)|\ge 3$ and every $6$-cycle bounds a face,
and let $Z$ be either a vertex or a $6$-cycle in $G$.
Let us assign the initial charge $\ch_0(v)=2\deg(v) - 6$ to each vertex $v\in V(G)\setminus V(Z)$ and
$\ch_0(f)=|f|-6$ to each face $f$ of $G$.  If $Z$ consists of a single vertex $z$, then let $\ch_0(z)=2\deg(z)$,
otherwise let $\ch_0(z)=2\deg(z)-4$ for each $z\in V(Z)$.  Let $\delta=6$ if $|V(Z)|=1$ and $\delta=12$ if $|V(Z)|=6$.
By Euler's formula, we have
\begin{align*}
\sum_{v\in V(G)} \ch_0(v)+\sum_{f\in F(G)}\ch_0(f)&=(4|E(G)|-6|V(G)|)+\delta+(2|E(G)|-6|F(G)|)\\
&=6(|E(G)|-|V(G)|-|F(G)|)+\delta=\delta-12.
\end{align*}

For a face $f$ of $G$ and a vertex $u$, let us define $d_Z(f,u)$ as $0$ if $u$ is incident with $f$,
and as the minimum of $d_Z(v,u)$ over all vertices $v\not\in V(Z)$ of degree three incident with $f$ otherwise.
Let us now redistribute the charge according to the following rules:
\begin{itemize}
\item[(T0)] Each vertex $z\in V(Z)$ sends 2 units of charge to each vertex $u\not\in V(Z)$ of degree two such that $d_Z(u,z)\le 7$.
\item[(T1)] Each vertex $v\not\in V(Z)$ of degree at least $4$ sends 1 unit of charge to each vertex $u\not\in V(Z)$ of degree two such that $d_Z(u,v)\le 7$.
\item[(T2)] Each face $f$ of length at least $7$ sends 1 unit of charge to each vertex $u\not\in V(Z)$ of degree two such that
$u$ receives at most one unit of charge by (T0) and (T1) and $d_Z(f,u)\le 3$.
\end{itemize}
Let $\ch_1$ denote the final charge after performing the redistribution according to these rules.

Whenever one of the rules (T0), (T1), or (T2) applies, this is because of some shortest $Z$-nice path $P$ ending in a vertex $u\not\in V(Z)$ of degree two
and starting in a vertex of $Z$, or a vertex of degree at least four, or a vertex incident with a $(\ge\!7)$-face.
This path is not necessarily unique, but for each application of a rule we fix one such path $P$ arbitrarily; and if $P$ has length at least one
and starts with an edge $e$, we say that the charge to $u$ is being sent \emph{through} the edge $e$.

\begin{lemma}\label{lemma-finsix}
Let $G$ be a $2$-connected plane graph of girth at least six such that $|V(G)|>2$ and every $6$-cycle bounds a face,
and let $Z$ be either a vertex or a $6$-cycle in $G$.
If $G$ does not contain a $3$-reducible induced subgraph with at most $29$ vertices disjoint from $Z$, then
$\ch_1(v)\ge 0$ for all $v\in V(G)$ and $\ch_1(f)\ge 0$ for all $f\in F(G)$.  If additionally $Z$ is a $6$-cycle, then
some vertex or face has positive final charge.
\end{lemma}
\begin{proof}
Let $f$ be a face of $G$.  Since $G$ is $2$-connected, $|V(G)|>2$, and $G$ has girth at least six,
we have $|f|\ge 6$.  If $|f|=6$, then $\ch_1(f)=\ch_0(f)=0$.  Hence, suppose that $|f|\ge 7$.
Let $W$ denote the cycle bounding $f$, and consider a vertex $x\in W$, contained in a subpath
$v_1v_2v_3xv_4v_5v_6$ of $W$.  If $v_3$ and $v_4$ do not belong to $Z$ and have degree three, then let $W_x=v_2v_3xv_4v_5$;
if $v_3$ belongs to $Z$ or has degree other than three, then $W_x=v_3xv_4v_5v_6$; otherwise, $v_4$ belongs to $Z$
or has degree other than three, and we let $W_x=v_1v_2v_3xv_4$.
For each vertex $u\not\in V(Z)$ of degree two to which $f$ sends charge according to (T2), let $u'=u$ if $u$ is incident with $f$ and
let $u'\not\in V(Z)$ be a vertex of degree three incident with $f$ such that $d_Z(u,u')\le 3$ otherwise.
Observe that all internal vertices of $W_{u'}$ other than $u'$ have degree three and no vertex of $W_{u'}$ belongs to $Z$,
as otherwise either $G$ contains a $3$-reducible induced subgraph with at most $6$ vertices disjoint from $Z$ by Lemma~\ref{lemma-dist2},
or $u$ receives $2$ units of charge by (T0) and (T1).
Furthermore, if $f$ sends charge to distinct vertices $u_1$ and $u_2$ by (T2), then the paths $W_{u'_1}$ and $W_{u'_2}$ are edge-disjoint,
as otherwise $G$ would contain a $3$-reducible induced subgraph with at most $12$ vertices disjoint from $Z$ by Lemma~\ref{lemma-dist2}.
We conclude that the amount of charge sent by $f$ is at most $\lfloor |f|/4\rfloor$,
and thus $\ch_1(f)\ge |f|-6-\lfloor |f|/4\rfloor=\lceil (3|f|-24)/4\rceil\ge 0$.

Furthermore, if $Z$ shares an edge with $f$, then the amount of charge sent by $f$ is at most $\lfloor (|f|-3)/4\rfloor$ by the
same argument, and thus $\ch_1(f)\ge \lceil (3|f|-21)/4\rceil$.  If $|f|\ge 8$, this implies $\ch_1(f)>0$.  If $|f|=7$,
then observe that $f$ does not send charge by (T2): if it sent charge to some vertex $u$, then the endvertices of $W_{u'}$ would
have distinct neighbors in $Z$, and thus $u$ would receive at least $2$ units of charge by (T0) and (T1), which is a contradiction.
We conclude that if a face $f$ of length at least $7$ shares an edge with $Z$, then $\ch_1(f)>0$.

Let $v$ be a vertex of $G$.  Since $G$ is $2$-connected and $|V(G)|>2$, we have $\deg(v)\ge 2$.
Let us first consider the case that $v\in V(Z)$ or $\deg(v)\ge 4$.
Let $m$ denote the number of vertices $u\not\in V(Z)$ of degree two such that $d_Z(u,v)\le 7$.
Note that $v$ does not send charge to two distinct vertices $u_1$ and $u_2$ through the same edge,
as otherwise we would have $d_Z(u_1,u_2)\le 12$ and by Lemma~\ref{lemma-dist2} $G$ would contain
a $3$-reducible induced subgraph disjoint from $Z$ with at most $13$ vertices.
Hence, we have $m\le \deg(v)$.  Furthermore, no charge is sent through the edges of $Z$, and
thus if $v\in Z$ and $Z$ is a $6$-cycle, then $m\le \deg(v)-2$.
Finally, if $v\not\in V(Z)$ and $\deg(v)\le 5$, then Lemma~\ref{lemma-distk} implies $m\le \deg(v)-2$,
as otherwise $G$ would contain a $3$-reducible induced subgraph disjoint from $Z$ with at most $29$ vertices.
Hence, if $Z$ consists just of the vertex $v$, then $\ch_1(v)=\ch_0(v)-2m=2\deg(v)-2m\ge 0$ by (T0);
if $Z$ is a $6$-cycle and $v\in V(Z)$, then $\ch_1(v)=\ch_0(v)-2m=2\deg(v)-4-2m\ge 0$ by (T0);
if $v\not\in V(Z)$ and $\deg(v)\ge 6$, then $\ch_1(v)=\ch_0(v)-m=2\deg(v)-6-m\ge 0$ by (T1);
and if $v\not\in V(Z)$ and $\deg(v)\le 5$, then $\ch_1(v)=2\deg(v)-6-m\ge \deg(v)-4\ge 0$ by (T1).

Furthermore, consider the case that $Z$ shares an edge with a $6$-face $f$, bounded by a $6$-cycle $K$.
Since $G$ has girth at least $6$, $Z\cap K$ is a path; let $z_1$ and $z_2$ be its endvertices,
and let $e_1$ and $e_2$ be the edges of $E(K)\setminus E(Z)$ incident with $z_1$ and $z_2$, respectively.
If charge is sent through $e_1$ and $e_2$ to vertices $u_1$ and $u_2$, respectively, then either $u_1=u_2$,
or $u_1\neq u_2$ and the path $K-V(Z)$ contains a vertex $w$ of degree at least $4$ by Lemma~\ref{lemma-dist2}.
In the former case, $\ch_1(u_1)\ge\ch_0(u_1)+4>0$.  In the latter case we can by symmetry assume that $w$
is at distance two from $z_1$, and thus $u_1$ also receives charge from $w$ by (T1) and
$\ch_1(u_1)\ge\ch_0(u_1)+3>0$.

We conclude that if $Z$ shares an edge with a $6$-face $f$,
then $G$ contains a vertex $u$ with $\ch_1(u)>0$ (with $u\in\{u_1,u_2\}$ if charge is sent through both $e_1$ and $e_2$,
and $u\in \{z_1,z_2\}$ otherwise).  Together with the preceding analysis of the case that a $(\ge\!7)$-face shares
an edge with $Z$, this implies that if $Z$ is a $6$-cycle, then some vertex or face has positive final charge.

Finally, let us consider the final charge of a vertex $v\not\in V(Z)$ of degree at most three.
If $\deg(v)=3$, then $\ch_1(v)=\ch_0(v)=0$; hence, suppose $\deg(v)=2$. 
If $v$ receives charge by (T0), then $\ch_1(v)\ge\ch_0(v)+2\ge 0$.  Hence, we can assume this is not the case, and thus $d_Z(v,z)>7$
for every $z\in V(Z)$.
Let $f_1$ and $f_2$ be the faces incident with $v$.  If $v$ receives charge neither from $f_1$ and $f_2$ by (T2) nor from the vertices
incident with $f_1$ or $f_2$ by (T1), then both $f_1$ and $f_2$ would have length six, and either $G$ would contain a $3$-reducible
induced subgraph with at most $4$ vertices disjoint from $Z$ by Lemma~\ref{lemma-dist2}, or all all other vertices incident with $f_1$
and $f_2$ would have degree exactly three and would not belong to $Z$.
But in the latter case, $G$ would contain a $3$-reducible induced subgraph with at most $10$ vertices disjoint from $Z$ by Lemma~\ref{lemma-2faces}.
This is a contradiction, and thus $v$ receives at least one unit of charge from $f_1$, $f_2$, or their incident vertices.

If $v$ receives charge from at least two vertices or faces, then $\ch_1(v)\ge\ch_0(v)+2\ge 0$.  Hence, we can assume that
$v$ receives charge either from exactly one vertex or from exactly one face.  In particular, we can by symmetry assume that $|f_2|=6$;
let $vv_1v_2v_3v_4v_5$ be the cycle bounding $f_2$.  By Lemma~\ref{lemma-dist2}, the neighbors of $v$ have degree
at least three.  Suppose now that $v$ does not receive charge from either of its
neighbors. We can by symmetry assume that $v$ receives charge from $f_1$ or a vertex $x$ incident with $f_1$.
In the latter case, we have $|f_1|=6$, and since $G$ has girth at least six, we conclude that
$x$ is not incident with $f_2$.  Consequently, Lemma~\ref{lemma-dist2} implies that all vertices incident with $f_2$ other than $v$
have degree exactly three.  Let $v'_2$ be the neighbor of $v_1$ distinct from $v$ and $v_2$.  By symmetry, we can assume that $v'_2\neq x$,
and thus $v'_2$ has degree exactly three.  Let $f_3$ be the face incident with the path $v_2v_1v'_2$.  Since $v$ does not receive
charge from $f_3$, we have $|f_3|=6$.  Since $G$ has girth at least six, $x$ is not incident with $f_3$, and since $v$ does not
receive charge from a vertex other than $x$, Lemma~\ref{lemma-dist2} implies that all vertices incident with $f_3$ have degree exactly three.
However, then $G$ would contain a $3$-reducible induced subgraph with at most $10$ vertices disjoint from $Z$ by Lemma~\ref{lemma-2faces}.

Therefore, we can assume that $v$ receives charge from its neighbor, by symmetry say $v_5$, and thus $|f_1|=6$, and by Lemma~\ref{lemma-dist2},
all vertices incident with $f_1$ or $f_2$ other than $v$ and $v_5$ have degree three.  Let $v'_2$ be the neighbor of $v_1$ distinct from $v$ and $v_2$,
let $f_3$ denote the face incident with the path $v_2v_1v'_2$, and let $f_4$ and $f'_4$ be the faces incident with $v_2$ and $v'_2$, respectively,
distinct from $f_1$, $f_2$, and $f_3$.  Since $v$ does not receive charge from $f_3$, $f_4$, or $f'_4$, we conclude these faces have length
six, and since $G$ has girth at least six, $v_5$ is not incident with $f_3$, $f_4$, or $f'_4$.  Consequently, all vertices incident with these
faces have degree three.  Let $v_2v_6$ and $v'_2v'_6$ be the edges $f_3$ shares with $f_4$ and $f'_4$,
and let $f_5$ and $f'_5$ be the faces incident with $v_6$ and $v'_6$, respectively, distinct from $f_3$, $f_4$, and $f'_4$.
Since $v$ does not receive charge from $f_5$ and $f'_5$, both faces have length six.  Let $v_7v_8$ be the edge shared by $f_5$ and $f'_5$.
Since $G$ has girth at least six and every $6$-cycle in $G$ bounds a face, we conclude $v_7\neq v_5\neq v_8$,
and consequently $v_5$ is incident with at most one of the faces $f_5$ and $f'_5$.  By symmetry, assume that $v_5$ is not incident with
$f_5$.  By Lemma~\ref{lemma-dist2} and the assumption that $v$ receives charge only from $v_5$, we conclude that all vertices incident with $f_5$
have degree three.  However, then $G$ contains a a $3$-reducible induced subgraph with at most $13$ vertices disjoint from $Z$ (induced by the vertices
incident with $f_3$, $f_4$, and $f_5$) by Lemma~\ref{lemma-3faces}.
This is a contradiction, showing that $\ch_1(v)\ge 0$.
\end{proof}

We are now ready to prove the main result of this note.

\begin{proof}[Proof of Theorem~\ref{thm:sixGirth}]
Let $G_0$ be a plane graph of girth at least six.  We apply Lemma~\ref{lemma-redu} (with $g=6$, $k=3$ and $b=29$) to show that $G_0$ is
weighted $\eps$-flexible (for fixed $\eps>0$ corresponding to the given values of $g$, $k$, and $b$)
with any assignment of lists of size $3$.  Since every subgraph of $G_0$ is planar and has girth at least six,
it suffices to prove that $G_0$ contains a $3$-reducible induced subgraph with at most $29$ vertices.
Without loss of generality, we can assume that $G_0$ is connected, and by Observation~\ref{obs-deg1}, we can assume
that all vertices of $G_0$ have degree at least two.

Let $G_1$ be a $2$-connected block of $G_0$; since $G_0$ has minimum degree at least two, we have $|V(G_1)|>2$.
If $G_1\neq G_0$, then let $Z_1$ be the subgraph of $G_1$ consisting of the vertex in that $G_1$ intersects the rest of the graph $G_0$;
otherwise, let $Z_1$ be an arbitrary subgraph of $G_1$ consisting of a single vertex.  Without loss of generality, we can assume that the vertex
of $Z_1$ is incident with the outer face of $G_1$.

If all $6$-cycles in $G_1$ bound faces, then let $G=G_1$ and $Z=Z_1$.  Otherwise, let $Z$ be a $6$-cycle in $G_1$ such that the open disk
$\Delta_Z$ bounded by $Z$ is not a face of $G_1$, and $\Delta_Z$ is minimal among the cycles $Z$ with this property.
Let $G$ be the subgraph of $G_1$ drawn in the closure of $\Delta_Z$.  Note that $G$ is $2$-connected and every $6$-cycle in $G$ bounds a face.

Let $\ch_0$ be the assignment of charges to vertices and faces of $G$ described at the beginning of the section.
Recall that the sum of these charges is negative if $|V(Z)|=1$ and $0$ if $Z$ is a $6$-cycle.  Redistributing the charge according to the rules (T0)---(T2)
gives us the charge assignment $\ch_1$ with the same sum of charges, and thus either some vertex or face has negative final charge,
or $Z$ is a $6$-cycle and all vertices and faces have final charge $0$.
Lemma~\ref{lemma-finsix} implies that $G$ contains a $3$-reducible induced subgraph $H$ disjoint from $Z$ with at most $29$ vertices.
Since $H$ is disjoint from $Z$, $H$ is also a $3$-reducible induced subgraph of $G_0$.
\end{proof}

\bibliographystyle{siam}
\bibliography{req-trfree}

\end{document}